\documentclass[11pt,letterpaper]{amsart}
\usepackage[english]{babel}
\usepackage{amsmath,amssymb}%,xcolor,latexsym,theorem}
\usepackage{comment}
\usepackage{enumitem}
\usepackage{amsthm}
\usepackage{xcolor}
\usepackage{amscd}
\usepackage{graphicx}
\newtheorem{theorem}{Theorem}
\newtheorem{lemma}[theorem]{Lemma}

\newtheorem{corollary}[theorem]{Corollary}
\newtheorem{proposition}[theorem]{Proposition}
\newtheorem{conjecture}[theorem]{Conjecture}

\newtheorem*{defn}{Definition}

\newtheorem*{remark}{Remark}

\usepackage{blindtext}

\title{Cellular waists of hyperbolic spaces}

\author{Grigori Avramidi and Thomas Delzant}

\input xy
\xyoption{all}
\begin{document}
\begin{abstract}
We find lower bounds on the topological complexity of fibers of PL and generic smooth maps $p:M^d\rightarrow\mathbb R^m$, where $M^d$ is a closed hyperbolic manifold of large injectivity radius. More precisely, we show that if the injectivity radius of $M$ is greater than $50\log((n+1)!)$, then for each dimension $0<k<d-m$ there is a point $z\in\mathbb R^m$ such that any cell structure on the fiber $p^{-1}(z)$ has more than $n$ cells of dimension $k$. The proof is based on the freedom theorem for ideals in group rings of hyperbolic groups proved in \cite{avramididelzanthyperbolic}.

%For a closed hyperbolic $d$-manifold $M^d$ of large injectivity radius and a PL map $p:M^d\rightarrow\mathbb R^m$ to a Euclidean space of dimension $m\leq d-2$, we use the freedom theorem for ideals proved in \cite{avramididelzanthyperbolic} to find a topologically complicated fiber. More precisely, we show that if the injectivity radius of $M$ is greater than $50\cdot\log((n+1)!)$, then for each $0<k<d-m$ there is a point $z\in\mathbb R^m$ such that any cell structure on the fiber $p^{-1}(z)$ has more than $n$ $k$-cells. 
\end{abstract}

\maketitle
\section{Introduction}
The goal of this paper is to obtain lower bounds on the topological complexity of a fiber of a map $M^d\rightarrow \mathbb R^m$ whenever $M$ is a closed hyperbolic manifold. Our study is motivated by a query of Guth (page 5 of \cite{guthwaist}) asking whether a map from a closed, arithmetic hyperbolic $7$-manifold to $\mathbb R^2$ must have a topologically complicated fiber.

The measures of topological complexity we use are minimum numbers of $k$-cells (for a fixed dimension $k$) in a cell structure. The main tool is the freedom theorem for ideals in group rings of hyperbolic groups we proved in \cite{avramididelzanthyperbolic}. In the special case of a %fundamental group of a closed 
hyperbolic manifold $M$ of large injectivity radius ($\geq 50\log((n+1)!)$) it shows, for a field $\mathbb K$, that $n$-generated ideals in the group ring $\mathbb K[\pi_1M]$ are free $\mathbb K[\pi_1M]$-modules. A ring with this property is called an {\it $n$-fir} (\cite{cohnbook}). %A ring all of whose ideals are free is called a free ideal ring---or {\it fir} for short--- and a ring all of whose $n$-generated ideals are free is called an {\it $n$-fir} (\cite{cohnbook}). 
The $n$-fir property leads to the following result.
%This property has topological consequences, e.g. the proof of Theorem 40 in \cite{avramididelzanthyperbolic} gives the following deformation result. %for any field $\mathbb K$.

%\subsection*{Notation:} Fix a field $\mathbb K$, a group $G$ and a cell complex model for $BG$. To simplify notation, $C_*(-)$ will stand for the cellular chain complex with coefficients in $\mathbb K[G]$. 

%\begin{remark}
%The bullet point implies $n$-generated submodules of free $\mathbb K[G]$-modules are free by the argument in Theorem 33.2 of \cite{avramididelzanthyperbolic}. 
%\end{remark}

% (see Theorem 40 in \cite{adhyperbolic}).
\begin{theorem}[see proof of Thm. 40 in \cite{avramididelzanthyperbolic}]
\label{essentialtheorem}
%For any $0<k<cd_{\mathbb K}(G)$,
Suppose $\mathbb K[G]$ is an $n$-fir.
If $X$ is a cell complex with $\leq n$ $k$-cells, then any continuous map $X\rightarrow BG$ is chain homotopic to a chain map $C_*(X;\mathbb K[G])\rightarrow C_*(BG;\mathbb K[G])$ that is zero above degree $k$.  
\end{theorem}
%\begin{remark}
%In this theorem and in the rest of the paper, we fix a cell complex model for $BG$.
%\end{remark}

In the present paper, we prove a version of this theorem for families.

\begin{theorem}
\label{maptheorem}
 Let $\mathbb K[G]$ be an $n$-fir. Let $p:X\rightarrow Z$ be a PL map of simplicial complexes, and fix an integer $k>0$. Suppose that for every point $z\in Z$, there is a cell structure on $p^{-1}(z)$ with at most $n$ $k$-cells.  Then for every continuous map 
 $$
 \phi:X\rightarrow BG
 $$ 
 there exists a chain map 
 $$
 C_*(X;\mathbb K[G])\rightarrow C_*(BG;\mathbb K[G])
 $$
 which is chain homotopic to the map induced by $\phi$ and which is zero in all degrees larger than $k+\dim Z$. %that is zero above degree $k+\dim (Z)$. 
\end{theorem}
\begin{remark}
In \cite{gromovexpanders}, section 2.4, Gromov proved, for $k=1$, an analogous result assuming that every $n$-generated subgroup of $G$ is a free group.
\end{remark}
Theorem \ref{essentialtheorem} lets one deform (on the chain complex level) a map $X\rightarrow BG$ into the $k$-skeleton if the domain has low $k$-dimensional complexity. Theorem \ref{maptheorem} lets one deform a $Z$-parameter family of such maps into the $(k+\dim(Z))$-skeleton. When there is an obstruction to such a deformation, we obtain a $k$-cellular ``waist inequality'' as a corollary.

%\footnote{The possibility of obtaining a cellular waist inequality for higher dimensional cells via the methods of \cite{avramididelzanthyperbolic} was suggested by Gromov in a talk he gave in Cambridge in 2025.}.
\begin{corollary}
\label{hypmanifoldcor}
Suppose $M^d$ is a closed hyperbolic $d$-manifold with injectivity radius $>50\log((n+1)!)$ and $p:M^d\rightarrow \mathbb R^m$ is a PL map. Then for each $0<k<d-m$ there is $z\in\mathbb R^m$ such that any cell structure on $p^{-1}(z)$ has more than $n$ $k$-cells. %Pick a triangulation $T$ of $\mathbb R^m$ and a cell structure on $M^d$ such that for every simplex $\sigma$ in $T$, $p^{-1}(St_{b^2T}(\nu_{\sigma}))$ is a subcomplex of $M$. Then, there is a simplex $\sigma$ in $T$ for which $p^{-1}(St_{b^2T}(\nu_{\sigma}))$ has more than $n$ $k$-cells for each $0<k<d-m$. 
\end{corollary}

\begin{proof}
The injectivity radius assumption implies that $n$-generated ideals in $\mathbb F_2[G]$ are free (Theorem 4 of \cite{avramididelzanthyperbolic}). Therefore, we can apply Theorem \ref{maptheorem} with $X=BG=M$ to deduce that if the conclusion of Corollary \ref{hypmanifoldcor} is not true, (i.e. if there is a $0<k<d-m$ such that for any $z\in\mathbb R^m$ there is a cell structure on $p^{-1}(z)$ with at most $n$ $k$-cells), then the identity map $\phi=id:M\rightarrow M$ is chain homotopic to a chain map that is zero above dimension $k+m$. In particular, it is zero in dimension $d$. Hence, the identity map induces the zero map on $H_d(M;\mathbb F_2)$, which is a contradiction since $M$ has a fundamental class.% in dimension $d$.
\end{proof}
\begin{remark}
 The possibility of obtaining a cellular waist inequality for higher dimensional cells via the methods of \cite{avramididelzanthyperbolic} was suggested by Gromov in a talk he gave in Cambridge in 2025\footnote{M. Gromov, Combinatorial and Homological Waists of Infinite Groups, https://www.newton.ac.uk/seminar/47815}. 
\end{remark}

\begin{remark}
The same result holds if $M^d$ is a $\mathbb K$-essential manifold (i.e. the image of $[M]$ in $H_*(B\pi_1M;\mathbb K)$ is not zero) if we assume that $\mathbb K[\pi_1M]$ is an $n$-fir. This is true if $\pi_1M$ acts on a $\delta$-hyperbolic space with displacement larger than $100\delta\log_2((n+1)!)$.
\end{remark}
\begin{remark}
Stable (generic) $C^{\infty}$ maps between manifolds are triangulable (\cite{verona}), so they satisfy the hypothesis of Theorem \ref{maptheorem}. 
\end{remark}
%The $k=1$ case of these types of theorems was proved by Gromov in Section 2.4 of \cite{gromovexpanders}. The current paper can be viewed as a group ring version of his approach. 

%Corollary \ref{hypmanifoldcor} addresses a query of Guth (page 5 of \cite{guthwaist}) asking whether a map $M^7\rightarrow\mathbb R^2$ from a closed, arithmetic hyperbolic $7$-manifold to $\mathbb R^2$ must have a topologically complicated fiber. 
%Gromov's result shows there is a fiber with high $1$-dimensional complexity, while our result provides (for each $0<k<d-m$) a fiber with high $k$-dimensional complexity.  

There are other recent papers on waist inequalities, e.g. for some locally symmetric spaces \cite{fraczyklowewaist} and Kazhdan groups \cite{badersauerwaist}. In that spirit, we would like to suggest a group ring theoretic conjecture that would imply (via the methods of this paper) $k$-cellular waist inequalities for maps $M^d\rightarrow\mathbb R^m$ from uniform locally symmetric spaces of real rank $r$ whenever $0<k<d-m-(r-1)$. 
\begin{conjecture}
\label{higherrankconjecture}
Suppose $\widetilde M^d$ is a symmetric space with curvature $-1\leq K\leq 0$ and real rank $r$. Then there is a function $f(n)$ such that if a group $G$ acts on $\widetilde M$ with minimum displacement greater than $f(n)$ ($d(gx,x)> f(n)$ for any $x\in\widetilde M$ and $1\not=g\in G$), then any $n$-generated ideal in $\mathbb K[G]$ has a free resolution of length $r-1$.%consisting of $r$ terms.
\end{conjecture}
Applying the conjecture to the augmentation ideal of $G$, we get a group theoretic special case of the conjecture: $n$-generated groups that act on $\widetilde M$ with minimum displacement $>f(n)$ should have cohomological dimension $\leq r$. In the real rank one (i.e. pinched negatively curved) case, the conjecture is true by the $\delta$-hyperbolic freedom theorem in \cite{avramididelzanthyperbolic}. But, even the group theoretic special case of it appears to be open for the rank two example $\widetilde M=\mathbb H^2\times\mathbb H^2$.

\subsection*{Strategy of the proof of Theorem \ref{maptheorem}}
The construction of the chain map in the statement of the theorem is organized into four steps, along the following lines. 

We have a simplicial map $p:X\rightarrow Z$. Let $Z_i$ be the $i$-skeleton of $Z$ and set $X_i=p^{-1}(Z_i)$. For each simplex $\sigma \subset Z$ choose a point $z\in\sigma^{\circ}$ and put $X_{\sigma}=p^{-1}(z)$. By hypothesis, each $X_{\sigma}$ admits a cell structure with at most $n$ $k$-cells.% of dimension $k$. 

\subsubsection*{First step (section \ref{sqf})} The first step is topological. Using the given cell structures on the fibres, we replace $X$ by a cell complex $X'$ obtained by gluing the cell complexes $\sigma\times X_{\sigma}$. The following properties of $X'$ will be used. 

The cell complex $X'$ is filtered by subcomplexes $X_0=X_0'\subset X_1'\subset\dots$ such that 
\begin{enumerate}
\item 
$X_0'=X_0=\coprod_{v\in Z_0}v\times X_v$,
\item
$X'_i\setminus X'_{i-1}=\coprod_{|\sigma|=i}\sigma^{\circ}\times X_{\sigma}$,
\item 
$X'$ is homotopy equivalent to $X$ relative to $X_0.$
\end{enumerate}
The significance of this complex is that its cellular structure (and the associated chain complex) can be described simply in terms of the cellular structures of the fibres $X_{\sigma}$ of the map $p$. 
\subsubsection*{Second step (section \ref{chaintruncation})} For each vertex $v$ of $Z$, we construct a chain map $f_v:C_*(X_v;\mathbb K[G])\rightarrow C_*(BG;\mathbb K[G])$ which coincides with $\phi$ below degree $k$ and vanishes above degree $k$. This is possible because, for every vertex $v$, the fiber of $v$ has no more than $n$ cells in dimension $k$, and $\mathbb K[G]$ is an $n$-fir. 
\subsubsection*{Third step (section \ref{chainmapF})} We %explain how to 
extend these local maps on vertices to a global chain map 
$$
F:C_*(X';\mathbb K[G])\rightarrow C_*(BG;\mathbb K[G])
$$
that is zero in all dimensions $>k+\dim(Z)$, and coincides with $f_v$ on each $C_*(X_v)$.% for every vertex $v$ of $Z$.
\subsubsection*{Last step (section \ref{proofoftheorem})} We use the homotopy equivalence $L$ (relative to $X_0$) from $X$ to $X'$ and the asphericity of $BG$ to prove that the chain map induced by $\phi$ is chain homotopic to the map $FL$, which is zero on $C_{>k+\dim(Z)}(X)$.

\subsection*{Examples} We end the paper with a series of examples illustrating the difference between simplicial and cellular complexity (\ref{simpvscell}), and showing that the lower bound on injectivity radius of $M$ in Corollary \ref{hypmanifoldcor} cannot be replaced with a lower bound on volume (\ref{largevol}) or the number of cells of $M$ (\ref{largerank} and \ref{complex}).

\section{Cellular singular fibration\label{sqf}}
In this paper, we are primarily interested in understanding the topological complexity of fibres of PL maps. The features of PL maps that will be relevant to us are captured in what we call a {\it cellular singular fibration}. We will first give the definition and then prove that PL maps between simplicial complexes satisfy it. 
\subsection{A definition}
%\begin{defn}[Cellular singular fibration]
Let $X$ be a topological space, $Z$ a cell complex and $p:X\rightarrow Z$ a continuous map. Let $Z_i$ be the $i$-skeleton of $Z$ and set $X_i:=p^{-1}(Z_i)$. We say that $p$ is a {\it cellular singular fibration} if the following conditions hold. 
\begin{enumerate}
\item 
\label{baryfib}
For every point $z\in Z$, the fiber $p^{-1}(z)$ has a cell structure.
\item 
\label{trivfiber}
For every cell $\sigma$ in $Z$, $p$ is a trivial fibration over the open cell $\sigma^{\circ}$. Denote the fiber of this fibration $X_{\sigma}$, so that $p^{-1}(\sigma^\circ)$ is homeomorphic to $\sigma^\circ\times X_{\sigma}$.  
\item 
\label{mappingcylinderneighborhood}
For each $i$, we have a map%$X_{i-1}$ in $X_i$ has a mapping cylinder neighborhood $M(\varphi_i)$, where 
$$
\varphi_i:\coprod_{|\sigma|=i}S^{i-1}\times X_{\sigma}\rightarrow X_{i-1}
$$
%\textcolor{blue}{
%What we actually use in Lemma \ref{homotopyeq} (and prove for PL maps in Proposition \ref{PLmap}) is that 
and a decomposition 
$$
X_i=X_{i-1}\cup_{\varphi_i}\left(\coprod_{|\sigma|=i}D^i\times X_{\sigma}\right).
$$ 
\item 
\label{regnbhd}
For each $i$, $X_0$ has a mapping cylinder neighborhood in $X_i$.
\end{enumerate}
%Then we will call $p$ a cellular singular fibration.
\begin{remark}
Part (\ref{mappingcylinderneighborhood}) of the definition implies that $X_{i-1}$ in $X_i$ has a mapping cylinder neighborhood (homeomorphic to the mapping cylinder $M(\varphi_i)$ of $\varphi_i$.)
\end{remark}
\begin{remark}
This is a variation of the definition of a (stratified singular) quasi-fibration on p.762 of \cite{gromovexpanders}. Here is the precise definition of Gromov.

Let $p:X\rightarrow Z$ be a continuous map of topological spaces. It is a stratified singular quasi-fibration if the following holds. There is a filtration $Z_0\subset Z_i\subset\dots$ of $Z$ by closed cellular subsets (e.g. $Z_i$ is the $i$-skeleton of some cell structure on $Z$) such that 
\begin{itemize}
\item 
the map $p$ is a fibration with connected fibres over all open strata $Z_i\setminus Z_{i-1}$ (the fibres may be different over different connected components of $Z_i\setminus Z_{i-1}$),
\item 
there are homotopies of the identity maps $h_i(\cdot,t):Z_i\times[0,1]\rightarrow Z_i$ such that $h_i(\cdot,1)$ sends some neighborhood of $Z_{i-1}\subset Z_i$ to $Z_{i-1}$ and such that these homotopies lift to homotopies of the identity maps of $p^{-1}(Z_i)$.  
\end{itemize}
`Cellular singular fibration' is a more restrictive notion requiring, in particular, that $X_{i-1}$ has a mapping cylinder neighborhood in $X_i$. The extra structure will be used to describe, given a cell structure on each $X_{\sigma}$, a cell structure on a space homotopy equivalent to $X$.
\end{remark}
\subsection{PL maps}
The main example is given by the next proposition. 
\begin{proposition}
\label{PLmap}
Let $X$ and $Z$ be two simplicial complexes. Any PL map $X\rightarrow Z$ is a cellular singular fibration. 
\end{proposition}
Since stable $C^{\infty}$ maps are triangulable (\cite{verona}), this gives the following.% corollary. 
\begin{corollary}
A stable $C^{\infty}$ map $X\rightarrow Z$ between manifolds is a cellular singular fibration provided $X$ is compact. 
\end{corollary}
\begin{proof}[Proof of Proposition \ref{PLmap}]

Let $p:X\rightarrow Z$ be a piecewise linear map of simplicial complexes. By definition of a PL map, after subdividing $X$ and $Z$ we may assume that $p$ is simplicial. Let $Z_i$ be the $i$-skeleton of $Z$ and $X_i=p^{-1}(Z_i)$. Observe several features of the simplicial setting. Below, the letter $b$ applied to a simplicial complex denotes its barycentric subdivision.

First, for each simplex $\sigma \in Z$, the inverse image of the interior $p^{-1}(\sigma^{\circ})$ is a trivial fibration over $\sigma^\circ$ since for each simplex $\tau$ in $X$ which has $\tau^{\circ}\subset p^{-1}(\sigma^{\circ})$, the restriction $p\mid_{\tau^{\circ}}$ has constant rank $|\sigma|$ (is an affine submersion), so (\ref{trivfiber}) is satisfied. 

Second, $p:bX\rightarrow bZ$ is simplicial, and the fiber over a point in $\sigma^{\circ}$ can be identified with the inverse image of the barycenter $p^{-1}(b_{\sigma})$, which is a subcomplex of $bX$. So, (\ref{baryfib}) is satisfied.

%Second, the fiber can be identified with the inverse image $p^{-1}(b_{\sigma})$ of the barycenter $b_{\sigma}$ of $\sigma$, and hence it is a subcomplex of the barycentric subdivision $bX$ of $X$. More precisely, $p:bX\rightarrow bZ$ is simplicial, and $p^{-1}(b_{\sigma})$ is a subcomplex of $bX$, so %More precisely (see Lemma 2.16 in \cite{rourkesanderson}) there exists a subdivision $X'$ of $X$ 

Third, the regular neighborhood of $bX_{i-1}$ in $bX_i$ deformation retracts onto $bX_{i-1}$, and this deformation retraction covers the deformation retraction of the regular neighborhood of $bZ_{i-1}$ in $bZ_i$ onto $bZ_{i-1}$. The regular neighborhood of $bZ_{i-1}$ in $bZ_i$ is the mapping cylinder of a retraction $r_i:\coprod_{|\sigma|=i}S_{\sigma}^{i-1}\rightarrow bZ_{i-1}$ of a disjoint union of $(i-1)$-spheres (for each $i$-simplex $\sigma$ in $Z$, we pick a closed $i$-disk $D^i_{\sigma}$ embedded in the interior $\sigma^{\circ}$ of $\sigma$, and let $S^{i-1}_{\sigma}:=\partial D^i_{\sigma}$), and $bZ_i$ can be described as union 
$$
bZ_i=bZ_{i-1}\cup_{r_i}\left(\coprod_{|\sigma|=i}D_{\sigma}^i\right).
$$
The retraction $r_i$ is covered by a retraction $\varphi_i:\coprod_{|\sigma|=i}p^{-1}(S_{\sigma}^{i-1})\rightarrow bX_{i-1}$, the regular neighborhood of $bX_{i-1}$ in $bX_i$ is the mapping cylinder $M(\varphi_i)$ of this retraction, and $bX_i$ can be described as a union 
$$
bX_i=bX_{i-1}\cup_{\varphi_i}\left(\coprod_{|\sigma|=i}p^{-1}(D_{\sigma}^i)\right).
$$
Since $p$ is a trivial fibration over each $\sigma^{\circ}$, it is trivial over $D_{\sigma}^i\subset\sigma^{\circ}$, so (\ref{mappingcylinderneighborhood}) is satisfied. 

Finally, fourth, $X_0$ is a subcomplex of $X_i$ so it has a regular neighborhood in $X_i$, which implies (\ref{regnbhd}).
\end{proof}

%For the proof of Theorem \ref{maptheorem}, we will need to put arbitrary cell structures on fibres of $p$. These may not be compatible inside of $X$, so we will need to homotope attaching maps to be cellular and build a new, homotopy model for $X$ in which they are compatible. An appropriate setting for this (somewhat more general than the situation of a PL map) is the following. 

%\begin{remark}
%Does the fourth bullet follow from the third one?
%\end{remark}

%We will say that $p$ is a {\it cellular} stratified quasi-fibration if every fiber $X_{\sigma}$ can be given the structure of a cell complex.
%\subsection*{Examples}
%We observed above that a PL map of simplicial complexes (which we may take to be simplicial after appropriate subdivisions) satisfies the definition. Real algebraic or stable $C^{\infty}$ maps are triangulable (\cite{hardt,verona}), so they also satisfy it. %these conditions. %A generic smooth or generic real analytic map does, as well.  a generic smooth map\footnote{Every smooth manifold can be triangulated. Is it not true that every smooth map of smooth manifolds is PL for some triangulations on the domain and range?}, then it is a stratified quasi-fibration and every fiber can be given the structure of a cell complex.

%\begin{remark}
%The PL example has the additional property (which we want) that a regular neighborhood of $X_{i-1}$ in $X_i$ is homotopy equivalent to $X_{i-1}$. Does this follow from the definition of stratified quasi-fibration above, or does it need to be an additional assumption?
%\end{remark}

\subsection{A cell complex model for a cellular singular fibration}
Let $p:X\rightarrow Z$ be a cellular singular fibration. For each cell $\sigma$ in $Z$ we fix a cellular structure on the fiber $X_{\sigma}=p^{-1}(b_{\sigma})$. In this subsection we construct an auxiliary space $X'$ obtained by gluing the spaces $\sigma\times X_{\sigma}$. The space $X'$ is homotopy equivalent to $X$  
%Next, we show that the above four conditions are enough to replace $X$ by a (filtered) cell complex $X'$ that has $X'_0=X_0=\coprod_v v\times X_v$ (with specified cell structures) and is homotopy equivalent to $X$ 
relative (see the definition below)
%\footnote{Suppose that $X$ and $X'$ both contain $X_0$. Then $L:X\rightarrow X'$ is a homotopy equivalence relative to $X_0$ if $L$ fixes $X_0$ and there is $R:X'\rightarrow X$ that fixes $X_0$ such that $RL$ is homotopic to $id_X$ and $LR$ is homotopic to $id_{X'}$, both via homotopies that fix $X_0$.} 
to $X_0$ and has a cell structure built out of the cell structures on the $X_{\sigma}$. 

\begin{defn}
Let $X$ and $X'$ be topological spaces, both containing a subspace $X_0$. A map $L:X\rightarrow X'$ is a homotopy equivalence relative to $X_0$ if $L$ fixes $X_0$ and there is $R:X'\rightarrow X$ that fixes $X_0$ such that $RL$ is homotopic to $id_X$ and $LR$ is homotopic to $id_{X'}$, both via homotopies that fix $X_0$.
\end{defn}

\begin{proposition}
\label{cellularprop}
Suppose that $p:X\rightarrow Z$ is a cellular singular fibration. %\footnote{In particular, this applies if $p$ is a simplicial map of simplicial complexes.}. 
For each cell $\sigma$ in $Z$, pick a cell structure on $X_\sigma$. Then, there is a cell complex $X'$ with a filtration by subcomplexes $X'_0\subset X'_1\subset\dots$ such that 
\begin{itemize}
%\item 
%for each $z\in Z$, $p^{-1}(z)$ is homeomorphic to $p'^{-1}(z)$,
%\item 
%each $X'_i$ is a subcomplex, 
\item 
$X'_0=X_0=\coprod_v v\times X_v$, %equipped with the specified cell structures,
\item 
$X'_i\setminus X'_{i-1}=\coprod_{|\sigma|=i}\sigma^{\circ}\times X_{\sigma}$, and
\item 
$X'_i$ is homotopy equivalent to $X_i$ relative to $X_0$. 
\end{itemize}
%homotopy equivalent to $X$ and a stratified quasi-fibration $p':X'\rightarrow Z$ with the same fibres as $p$.\footnote{We are not claiming that there is a fiberwise homotopy equivalence.}
\end{proposition}

%\begin{remark}
%The virtue of the above definition of cellular, stratified quasi-fibration is that it does not, a priori, posit any compatibility between the cell structures on the different fibres. We pick any cell structures on the fibres and build a new homotopy model $X'$ for $X$ on which these cell structures are compatible. 
%\end{remark}
\begin{proof}
We will construct the filtered complex $X'$ by induction on the filtration. To start, set $X'_0:=X_0$ which is the disjoint union $\coprod_v v\times X_v$ where the union is over the vertices of $Z$ and each $X_v$ is equipped with the chosen cell structure. Suppose we have constructed a cell complex $X'_{i-1}$, and a homotopy equivalence $L_{i-1}:X_{i-1}\rightarrow X'_{i-1}$ that is the identity on $X_0$. % and a map $p':X'_{i-1}\rightarrow Z$. 
Now, set
$$
X'_i:=X'_{i-1}\cup_{\varphi'}\left(\coprod_{|\sigma|=i}D^i\times X_{\sigma}\right)
$$
where %the union is over all the $i$-cells in $Z$, 
\begin{itemize}
\item 
$D^i$ is given the cell structure with one $0$-cell, one $(i-1)$-cell, and one $i$-cell, 
\item 
$D^i\times X_{\sigma}$ is equipped with the product of this structure and the chosen cell structure on $X_{\sigma}$, and
\item 
the gluing is done via a cellular approximation $\varphi'$ of the map 
$$
\coprod_{|\sigma|=i}(S^{i-1}\times X_{\sigma})\stackrel{\varphi_i}\longrightarrow X_{i-1}\stackrel{L_{i-1}}\longrightarrow X'_{i-1}.
$$
\end{itemize}
The next lemma finishes the inductive step and hence the proof of the proposition. 
\begin{lemma}
\label{homotopyeq}
The space $X_i$ is homotopy equivalent to $X_i'$ relative to $X_0$.% Moreover, we may pick $L_i$ and a homotopy inverse $R_i:X_i'\rightarrow X_i$ such that $L_i$ and $R_i$ restrict to the identity map on $X_0=X_0'$. 
\end{lemma}
\begin{proof}
To conserve notation, set $\Sigma:=\coprod_{|\sigma|=i}D^i\times X_{\sigma}$. Note that the mapping cylinder $M(L_{i-1})$ of $L_{i-1}:X_{i-1}\rightarrow X'_{i-1}$ deformation retracts onto $X'_{i-1}$ via a deformation retraction $r'_t$. Moreover, since $L_{i-1}$ is a homotopy equivalence, $M(L_{i-1})$ also deformation retracts onto $X_{i-1}$ via a deformation retraction $r_t$ (Corollary 0.21 of \cite{hatcherbook}). % has two deformation retractions: the usual one $r:M(L_{i-1})\rightarrow X'_{i-1}$ and also  
Now, we have a composition of homotopy equivalences
\begin{eqnarray*}
X_i&=&%M(\varphi_i)\cup_{\amalg S^{i-1}\times X_{\sigma}}\coprod D^i\times X_{\sigma}\\
%&\sim&M(\varphi_i)\cup_{\varphi_i}\coprod D^i\times X_{\sigma}\\
%&\sim&
X_{i-1}\cup_{\varphi_i}\Sigma\\%\coprod D^i\times X_{\sigma}\\
&\hookrightarrow&M(L_{i-1})\cup_{\varphi_i}\Sigma\\%\coprod D^i\times X_{\sigma}\\
&\sim&M(L_{i-1})\cup_{\varphi'}\Sigma\\%\coprod D^i\times X_{\sigma}\\
%&\sim&X'_{i-1}\cup_{L_{i-1}\circ\varphi_i}\coprod D^i\times X_{\sigma}\\
&\stackrel{r'_1}\rightarrow&X'_{i-1}\cup_{\varphi'}\Sigma\\
&=&X'_i.%\coprod D^i\times X_{\sigma}
\end{eqnarray*}
Here, the first inclusion is a homotopy equivalence since $r_t$ defines a deformation retraction from $M(L_{i-1})\cup_{\varphi_i}\Sigma$ to $X_{i-1}\cup_{\varphi_i}\Sigma$.
Since $\varphi_i$ is homotopic (in the mapping cylinder) to $L_{i-1}\circ\varphi_i$ which is homotopic to $\varphi'$, there is a homotopy equivalence $H$ relative to $M(L_{i-1})$ representing $\sim$.
%There is a map representing $\sim$ that is a homotopy equivalence that is the identity on $M(L_{i-1})$ since $\varphi_i$ is homotopic (in the mapping cylinder) to $L_{i-1}\circ\varphi_i$ which is homotopic to $\varphi'$ (Proposition 0.18 of \cite{hatcherbook}). 
Finally, $r_1'$ is a homotopy equivalence since $r'_t$ defines a deformation retraction from $M(L_{i-1})\cup_{\varphi'}\Sigma$ to $X'_{i-1}\cup_{\varphi'}\Sigma$. Note that $r_1'$ sends $X_0$ in $X_{i-1}$ identically to $X_0$ in $X_i'$ since $L_{i-1}$ is the identity on $X_0$. In summary, we obtain a composite homotopy equivalence $L_i:X_i\rightarrow X'_i$ that restricts to the identity on $X_0.$

To conclude that $L_i$ is a homotopy equivalence relative to $X_0$, observe that (by hypothesis) $X_0$ has a mapping cylinder neighborhood in $X_i$ and (by construction) $X_0$ is a subcomplex of $X'_i$. Therefore both pairs $(X_i,X_0)$ and $(X'_i,X_0)$ have the homotopy extension property (see Example 0.15 and Proposition 0.16 in \cite{hatcherbook}). So, Proposition 0.19 of \cite{hatcherbook} implies $L_i$ is a homotopy equivalence rel $X_0$. 
\end{proof}
\end{proof}
%So we can pick a homotopy equvalence $L_i:X_i\rightarrow X_i'$ and proceed to complete the inductive construction of $X'$. So, we arrive at the following.

%\subsection*{Reference map to classifing space $BG$}

\subsection{Cellular chain complex of $X'$}
Recall that $BG$ is an aspherical complex with fundamental group $G$. We fix a cell structure on it once and for all. We are given a map $\phi:X\rightarrow BG$, a cell complex $X'$ containing $X_0$ and two homotopy equivalences relative to $X_0$: $R:X'\rightarrow X$ and $L:X\rightarrow X'$ such that $RL$ is homotopic to the identity relative to $X_0$. 

Compose the map $\phi$ with a homotopy equivalence $R:X'\rightarrow X$ relative to $X_0=X'_0$ provided by Proposition \ref{cellularprop} and use the map $\phi R:X'\rightarrow BG$ to construct the cellular chain complex on $G$-covers of $X'_i$ (the fibre product $X'\times_{BG}EG$).

More algebraically, this is the cellular chain complex $C_*(X'_i;\mathbb K[G])$ with coefficients in the group ring of $G$, defined via the reference map $\phi R$. 

\subsection*{Notation:} To simplify notation, in the remainder of the paper for any cell complex $Y$ with reference map $\phi: Y\rightarrow BG$, we let $C_*(Y)$ stand for the cellular chain complex with coefficients in $\mathbb K[G]$. When we want to emphasize the reference map $\phi$ used in the construction of the chain complex, we will also write $C_*^\phi(Y)$. This will be relevant in the proof of Theorem \ref{maptheorem} in Section 5, where we compare two chain maps defined via different reference maps $\phi$ and $\phi RL$ from $X$ to $BG$.

\vspace{0.3cm}

On the level of cellular chain complexes, we have 
$$
C_*(X'_i)=C_*(X'_{i-1})\oplus\bigoplus_{|\sigma|=i}\sigma\times C_{*-i}(X_{\sigma}).
$$
So, as an abelian group $C_*(X'_i)=\oplus_{|\sigma|\leq i}\sigma\times C_{*-|\sigma|}(X_{\sigma})$ is generated by product cells. For an $i$-cell $\sigma$ in $Z$ and a chain $c\in C_*(X_{\sigma})$, the boundary map has the form 
%first write the boundary $\partial\sigma=\sum_\tau\sigma_\tau\cdot \tau$ as a sum of $(|\sigma|-1)$-cells in $Z$. Then 
\begin{equation}
\tag{*}
\label{leibnitzformula}
\partial(\sigma\times c)=\varphi(\partial\sigma\times c)+(-1)^{|\sigma|}\sigma\times\partial c
\end{equation}
where $\varphi:\partial\sigma\times X_{\sigma}\rightarrow X'_{i-1}$ is the attaching map. We may write  
$$
\varphi(\partial\sigma\times c)=\sum_{|\tau|<|\sigma|}\tau\times c_{\tau} 
$$
for $c_{\tau}\in C_*(X_{\tau})$. We will need the following observation.
\begin{comment}
\begin{remark}
Actually, we only end up using the cycle lemma, not the boundary lemma. 
\end{remark}
\begin{lemma}
\label{boundarylemma}
If $|\tau|=|\sigma|-1$ and $c$ is a boundary, then $c_{\tau}$ is a boundary. 
\end{lemma}
\begin{proof}
Suppose that $c=(-1)^{|\partial\sigma|}\partial b$. Then 
$$
\sum_{|\tau|<|\sigma|}\tau\times c_{\tau}=\varphi(\partial\sigma\times c)=\varphi(\partial(\partial\sigma\times b))
=\partial\varphi(\partial\sigma\times b)=\sum_{|\tau|<|\sigma|}\partial(\tau\times b_{\tau}).
$$
%Note that $\partial(\tau\times b_{\tau})=(-1)^{|\tau|}\tau\times \partial(b_{\tau})+\sum_{|\rho|<|\tau|}\rho\times (b_{\tau})_{\rho}$. 
If $|\tau|=|\sigma|-1$, compare $\tau$-coefficients to get $c_{\tau}=(-1)^{|\tau|}\partial(b_{\tau})$. So, $c_{\tau}$ is a boundary.  
\end{proof}
\end{comment}

\begin{lemma}
\label{cyclelemma}
If $|\tau|=|\sigma|-1$ and $c$ is a cycle, then $c_{\tau}$ is a cycle. 
\end{lemma}
\begin{proof}
Suppose $\partial c=0$. Then
$$
0=(-1)^{|\partial\sigma|}\varphi(\partial\sigma\times\partial c)=\partial\varphi(\partial\sigma\times c)=\sum_{|\tau|<|\sigma|}\partial(\tau\times c_{\tau}). 
$$
If $|\tau|=|\sigma|-1$, compare $\tau$-coefficients to get $0=(-1)^{|\tau|}\partial(c_{\tau})$. So, $c_{\tau}$ is a cycle.  
\end{proof}

%\begin{remark}
%Need to find a cleaner way to write the proof of this lemma...
%\end{remark}
%\begin{remark}
%The chain complex relation $\partial ^2=0$ implies that
%$h_{\rho\tau}\circ h_{\tau\sigma}=h_{\rho\tau'}\circ h_{\tau'\sigma}$.  
%\end{remark}
\section{Chain truncation on fibres\label{chaintruncation}}
In this section, suppose that $\mathbb K[G]$ is an $n$-fir. The $n$-fir property implies $n$-generated submodules of free $\mathbb K[G]$-modules are free by the argument in Theorem 33.2 of \cite{avramididelzanthyperbolic}. We will use this property to truncate cellular chain complexes, if the underlying cell complexes have at most $n$ cells of a given dimension $k>0$.
%Suppose that each $X_{\sigma}$ has at most $n$ $k$-cells. 
\begin{lemma}
\label{freelemma}
If $X_{\sigma}$ has at most $n$ $k$-cells, then $\partial C_{k}(X_\sigma)$ is a free $\mathbb K[G]$-module. 
\end{lemma}
\begin{proof}
The $\mathbb K[G]$-module $\partial C_{k}(X_{\sigma})$ is an $n$-generated submodule of the free module 
$C_{k-1}(X_\sigma)$, so it is free since $\mathbb K[G]$ is an $n$-fir. 
\end{proof}
Now, let $\phi_v$ be a cellular approximation of the restriction $\phi\mid_{X_v}:X_v\rightarrow BG$. We will use Lemma \ref{freelemma} to show that $\phi_v$ can be truncated in degree $k$ to a chain map that is zero on $C_{>k}$. The argument is the one given in Theorem 40 of \cite{avramididelzanthyperbolic}, so we will be brief. 
\begin{lemma}
\label{vertexchainmap}%[Theorem 40 in \cite{avramididelzanthyperbolic}]
%can be truncated in degree $k$ to a chain map that is zero on $C_{>k}$. %whose image lies in the $k$-skeleton of $BG$ 
%More precisely, 
Suppose $X_v$ has at most $n$ $k$-cells. Then, there is a map $$b_v:\partial C_{k}(X_v)\rightarrow C_{k}(BG)$$ such that %(obtained by lifting $\iota_{v}\circ\partial$, which can be done since $\partial C_k$ is free) such that 
\begin{eqnarray*}
f_v:C_*(X_v)&\rightarrow& C_*(BG)\\
c&\mapsto&\left\{
\begin{array}{ccc}
\phi_v(c)&\mbox{ if }&|c|<k,\\
b_{v}\partial (c)&\mbox{ if }&|c|=k,\\
0&\mbox{ if }&|c|>k,
\end{array}\right.
\end{eqnarray*}
is a chain map.
\end{lemma}
\begin{proof}
The image of the restriction $\phi_v:\partial C_k(X_v)\rightarrow C_{k-1}(BG)$ is in the image of $\partial$ since $\phi_v$ defines a chain map. Since $\partial C_k(X_v)$ is free, it has a lift $b_v:\partial C_k(X_v)\rightarrow C_{k}(BG)$. %Set $f_v=b_v\partial$. Then $f_v$ is a lift of $\phi\partial\mid_{X_v}$, i.e. 
%\begin{itemize}
%\item 
%$\partial f_v=\phi\partial\mid_{X_v}$,
%\end{itemize}
%and
%\begin{itemize}
%\item 
%$f_v\partial=b_v\partial^2=0$.
%\end{itemize}
For this $b_v$ the following ``truncation diagram'' commutes
$$
\begin{array}{ccccc}
C_{k+1}(X_v)&\stackrel{\partial}\rightarrow &C_k(X_v)&\stackrel{\partial}\rightarrow &C_{k-1}(X_v)\\
\downarrow&&\partial\downarrow&&||\hspace{0.2cm}\\
0&\rightarrow&\partial C_k(X_v)&\hookrightarrow&C_{k-1}(X_v)\\
\downarrow&&b_v\downarrow&&\downarrow\phi_v\\
C_{k+1}(BG)&\stackrel{\partial}\rightarrow &C_k(BG)&\stackrel{\partial}\rightarrow& C_{k-1}(BG),
\end{array}
$$
which implies that the map $f_v$ defined in the lemma is a chain map. 
\end{proof}

\section{A chain map on $X'$\label{chainmapF}}
%Suppose that for each $\sigma$, $X_\sigma$ has a cell structure with at most $n$ $k$-cells. Lemma \ref{vertexchainmap} produces a chain truncation $f_v$ of $C_*(X_v)$ above degree $k$.
%$$
%k:=\max_vk_v.
%$$
%Now, we will assemble these truncations into a chain map 
%$$
%F:C_*(X')\rightarrow C_*(BG)
%$$
%that is zero on $C_{>k+\dim(Z)}$.
In this section, we explain how to glue the local maps $f_v$ defined on $C_*(X_v)$ in the previous section to construct a global chain map on $C_*(X')$. Recall that $k$ is fixed once and for all and that we assume that every cell complex $X_{\sigma}$ has at most $n$ cells in dimension $k$. 

\begin{proposition}
\label{deform}
There exists a chain map $F:C_*(X')\rightarrow C_*(BG)$ that is zero in all dimensions $>k+\dim Z$, whose restriction to $C_*(X_v)$ is $f_v$ for each vertex $v$. 
\end{proposition}

Recall that, as an abelian group, 
$$
C_d(X')=\bigoplus_{|\sigma|+i=d}\sigma\times C_i(X_{\sigma}),
$$
represented by products of the form $\sigma\times c$ where $|\sigma|+|c|=d$. We will define $F$ by induction on $d$. Recall also that the {\it augmented chain complex} is obtained from $C_*(X')$ by adding $C_{-1}(X')=\mathbb K$, with the augmentation map $\varepsilon:C_0\rightarrow C_{-1}$ that sends each $0$-cell to $1\in\mathbb K$. 

Our argument repeatedly uses the following lifting property. 
\begin{lemma}
\label{lift}
A $\mathbb K[G]$-module map $\psi:V\rightarrow C_{d-1}(BG)$ lifts to $C_d(BG)$ if
\begin{itemize}
\item 
$V$ is a free $\mathbb K[G]$-module, and
\item 
$\varepsilon\psi=0$ (if $d=1$) or $\partial\psi=0$ (if $d>1$).
\end{itemize}
\end{lemma}
\begin{proof}
Since $BG$ is aspherical, the cellular chain complex on the universal cover of $BG$ gives a free $\mathbb K[G]$-resolution of $\mathbb K$
$$
\dots \stackrel{\partial}\rightarrow C_1(BG)\stackrel{\partial}\rightarrow C_0(BG)\stackrel{\varepsilon}\rightarrow \mathbb K\rightarrow 0.
$$
Therefore, the second bullet point lets us lift individual elements. Since the domain $V$ is free, we can apply this to a basis for $V$ to lift the map $\psi$.
\end{proof}

Proposition \ref{deform} is a consequence of the following, more detailed result. 
\begin{proposition}
\label{assembleprop}
There is a map $F:C_*(X')\rightarrow C_*(BG)$ of augmented %\footnote{The augmented chain complex has $C_{-1}(-)=\mathbb K$, with the boundary map $C_0\rightarrow C_{-1}$ given by the augmentation map $\varepsilon$ which sends each $0$-cell to $1\in\mathbb K$.} 
chain complexes such that 
\begin{enumerate}
\item 
$F\mid_{\mathbb K}=\phi\mid_{\mathbb K}$ in degree $-1$,
\item 
$F(v\times c)=f_v(c)$,
%\item 
%$$
%\partial F(\sigma\times c)=F(\partial\sigma\times c)+(-1)^{|\sigma|}F(\sigma\times\partial c),
%$$
\item
$F(\sigma\times c)=0$ if $|c|=k$ and $c$ is a cycle,
%if $|c|=k_{\sigma}$, then there is $b_{\sigma}:\partial C_{k_{\sigma}}(X_{\sigma})\rightarrow C_{k_{\sigma}+|\sigma|}(BG)$ such that 
%$$
%F(\sigma\times c)=b_{\sigma}\partial(c),$$ 
\item 
%if $|c|>k_{\sigma}$ then
$F(\sigma\times c)=0$ if $|c|>k$,
\item 
$\partial F(\sigma\times c)=F\partial(\sigma\times c)$ for $|\sigma\times c|\geq 1$, and 
\item 
$\varepsilon F(\sigma\times c)=F\varepsilon(\sigma\times c)$ for $|\sigma\times c|=0$.
\end{enumerate}
\end{proposition}
\begin{remark}
The second point says that the map $F$ restricts to $f_v$ on $C_*(X_v)$. The fourth point implies that $F$ is zero above dimension $k+\dim(Z)$. The third point is an auxiliary vanishing property that is used in order to carry out the inductive proof of the Proposition. The fifth and sixth points say that $F$ is a chain map of augmented chain complexes $F:C_*(X')\rightarrow C_*(BG)$. Therefore, this Proposition implies Proposition \ref{deform}.
\end{remark}
\subsection*{Proof of Proposition \ref{assembleprop}}
The first two points define $F$ on $C_{-1}(X')$ and $\oplus_vC_*(X_v)$. For $|\sigma|\geq 1$, we will define $F$ by induction on $d=|\sigma|+|c|$. 
\subsection*{Base case}
The base case is $d=0$, the case when $\sigma$ is a vertex $v$ and $c$ is a $0$-chain. We will do the more general case $v\times c$ where $v$ is a vertex and $c$ is arbitrary. We have to check that the four last properties listed in the Proposition are valid when $\sigma=v$ and we defined $F$ via 
%$$
%F(v\times c):=f_{v}(c),
%$$
$$
F(v\times c):=f_v(c)=\left\{
\begin{array}{ccc}
\phi_v(c)&\mbox{ if }&|c|<k,\\
b_{v}\partial (c)&\mbox{ if }&|c|=k,\\
0&\mbox{ if }&|c|>k.
\end{array}\right.
$$
Observe that,
\begin{itemize}
\item
$F(v\times c)=0$ if $|c|=k$ and $c$ is a cycle, 
%$F(v\times c)=b_v\partial(c)$ for $|c|=k$, and 
\item 
$F(v\times c)=0$ if $|c|>k$,
\end{itemize}
so the third and fourth points are satisfied. 

Moreover, Lemma \ref{vertexchainmap} shows that $F(v\times-)$ is a chain map, i.e. that $\partial F(v\times c)=F(v\times\partial c)$ and hence $\partial F(v\times c)=F\partial(v\times c)$ for $|c|\geq 1$. So, the fifth point is satisfied. Finally, for $|c|=0$ we have $\varepsilon F(v\times c)=\varepsilon\phi_v(c)=\phi_v(\varepsilon(c))=\phi_v(\varepsilon(v\times c))=F(\varepsilon(v\times c))$, so the sixth point is satisfied, as well.  

This finishes the proof of the base case. 

\subsection*{Inductive step}
At this point the map $F$ is already defined on all products $v\times c$, and if $d\geq 1$, $F(\sigma\times c)$ is also defined for $|\sigma|+|c|<d$.
%Suppose we have defined $F$ on all products of the form $v\times c$ and also all $\sigma\times c$ for $|\sigma|+|c|<d$. 
To complete the induction we need to define, for each simplex $\sigma$ of dimension $1\leq|\sigma|\leq d$, a map 
$$
F(\sigma\times -):C_{d-|\sigma|}(X_{\sigma})\rightarrow C_{d}(BG)
$$ 
satisfying the properties in Proposition \ref{assembleprop}. To that end, consider the map $$
F\partial(\sigma\times-):C_{d-|\sigma|}(X_{\sigma})\rightarrow C_{d-1}(BG),
$$
which is already defined since $F$ is already defined on $C_{d-1}(X')$. %As $d-|\sigma|<d$, we have already defined $F(\partial(\sigma\times-))$ on $C_{d-|\sigma|}(X_{\sigma})$.

\begin{lemma}
\label{vanishlemma}
The map 
$
F\partial(\sigma\times-)%:C_{d-|\sigma|}(X_{\sigma})\rightarrow C_{d-1}(BG).
$
satisfies the following properties.
\begin{enumerate}[label=$\mathrm{(\roman*)}$]
\item\label{cycle} 
If $d-|\sigma|=k$ and $c$ is a cycle in $C_{d-|\sigma|}(X_{\sigma})$, then $F\partial(\sigma\times c)=0$, 
\item\label{chainvanish} 
if $d-|\sigma|>k$ and $c$ is a chain in $C_{d-|\sigma|}(X_{\sigma})$, then $F\partial(\sigma\times c)=0$, 
\item\label{chain1} 
$\partial F\partial(\sigma\times-)=0$ for $d>1$, and
\item\label{chain2}
$\varepsilon F\partial(\sigma\times-)=0$ for $d=1$. 
\end{enumerate}
\end{lemma}
Let us assume this Lemma for the moment and use it to prove Proposition \ref{assembleprop}. There are three cases, depending on whether $d-|\sigma|$ is less than, equal to, or greater than $k$.

\subsection*{Case 1: $d-|\sigma|<k$} As $C_{d-|\sigma|}(X_{\sigma})$ is a free module, points \ref{chain1} (if $d>1$) and \ref{chain2} (if $d=1$) can be used to apply Lemma \ref{lift} and construct a lift of $F\partial(\sigma\times-):C_{d-|\sigma|}(X_{\sigma})\rightarrow C_{d-1}(BG)$ to $C_d(BG)$. % the third and fourth bullets\footnote{We use the third bullet when $d>1$ and the fourth when $d=1$.} (using Lemma \ref{lift} and the fact that $C_{d-|\sigma|}(X_{\sigma})$ is a free module) imply that $F(\partial(\sigma\times-))$ has a lift to $C_d(BG)$. 

We define $F(\sigma\times -)$ on $C_{d-|\sigma|}(X_{\sigma})$ to be such a lift. Then it satisfies
\begin{equation}
\tag{**}
\label{chain}
\partial F(\sigma\times -)=F\partial(\sigma\times -).
\end{equation}
%so, this choice of $F(\sigma\times-)$ satisfies all the conditions of the Proposition.

\subsection*{Case 2: $d-|\sigma|=k$} Point \ref{cycle} implies that the map $F\partial(\sigma\times -)$ vanishes on the kernel of $\partial$. Hence, this map factors through the image $\partial C_{d-|\sigma|}(X_{\sigma})$ of $\partial$. So, we can write it as a composition 
$$
F\partial(\sigma\times -)=\overline b_{\sigma}\circ\partial(-)
$$
of $\partial$ with a map $\overline b_{\sigma}:\partial C_{d-|\sigma|}(X_{\sigma})\rightarrow C_{d-1}(BG)$. Since $X_{\sigma}$ has at most $n$ $k$-cells, $\partial C_k(X_{\sigma})$ is free by Lemma \ref{freelemma}. Moreover, point \ref{chain1} implies $\partial\overline b_{\sigma}=0$. Therefore, Lemma \ref{lift} implies there is a map $b_{\sigma}:\partial C_{d-|\sigma|}(X_{\sigma})\rightarrow C_d(BG)$, that lifts $\overline b_{\sigma}$, i.e. such that $\partial b_{\sigma}=\overline b_{\sigma}$.

Set $F(\sigma\times-):=b_\sigma\circ\partial(-)$. This vanishes on cycles and satisfies (\ref{chain}).

\subsection*{Case 3: $d-|\sigma|>k$} Point \ref{chainvanish} implies that we can take $F(\sigma\times c):=0$. 

In each case, our choice of $F(\sigma\times -)$ satisfies the conditions of Proposition \ref{assembleprop}. So, to complete the inductive step (and the proof) it remains to prove Lemma \ref{vanishlemma}.

\begin{proof}[Proof of Lemma \ref{vanishlemma}]

%Choose such a lift and denote it by $F(\sigma\times-)$. Then 
%$$
%\partial F(\sigma\times-)=F(\partial(\sigma\times-)).
%$$
Recall the formula (\ref{leibnitzformula}): $$
\partial(\sigma\times c)=\varphi(\partial\sigma\times c)+(-1)^{|\sigma|}\sigma\times \partial c
$$ 
and the definition of the chains $c_{\tau}$ via $$\varphi(\partial\sigma\times c)=\sum_{|\tau|<|\sigma|}\tau\times c_{\tau}.
$$

Applying $F$ to the formula (\ref{leibnitzformula}) we get %for the differential of the product $\sigma\times c$, we get
$$
F\partial(\sigma\times c)=\sum_{|\tau|<|\sigma|-1}F(\tau\times c_{\tau})+\sum_{|\tau|=|\sigma|-1}F(\tau\times c_{\tau})+(-1)^{|\sigma|}F(\sigma\times\partial c). 
$$
We also know (by Lemma \ref{cyclelemma}) that the chains $c_{\tau}$ are cycles if $|\tau|=|\sigma|-1$. 

From this and the inductive hypothesis, we conclude that \begin{itemize}
\item 
$F\partial(\sigma\times c)$ vanishes if
$|c|=k$ and $c$ is a cycle: 
\begin{itemize}
\item 
$F(\tau\times c_{\tau})=0$ since $|\tau\times c_{\tau}|=d-1$ and 
\begin{itemize}
\item 
if $|\tau|=|\sigma|-1$ then $|c_{\tau}|=k$ and $c_{\tau}$ is a cycle by Lemma \ref{cyclelemma}, 
\item
if $|\tau|<|\sigma|-1$ then $|c_{\tau}|>k$. 
\end{itemize}
\item 
$F(\sigma\times\partial c)=0$ since $\partial c=0$.
\end{itemize}
\item 
$F\partial(\sigma\times c)$ vanishes if
$|c|>k$:
\begin{itemize}
\item 
$F(\tau\times c_{\tau})=0$ since $|\tau\times c_{\tau}|=d-1$ and $|c_{\tau}|>k$, and 
\item 
$F(\sigma\times\partial c)=0$ since $|\sigma\times\partial c|=d-1$ and either 
\begin{itemize}
\item 
$|\partial c|=k$ and $\partial c$ is a cycle, or
\item 
$|\partial c|>k$.
\end{itemize}
\end{itemize}
\end{itemize}
This establishes points \ref{cycle} and \ref{chainvanish} of the lemma. 

Finally, note that by the inductive hypothesis we have defined $F$ on $C_{\leq d-1}(X')$ and verified that it is a chain map on this complex. So, since $|\partial(\sigma\times c)|=d-1$, we have 
$$
\partial F\partial(\sigma\times c)=F\partial^2(\sigma\times c)=0
$$
for $d>1$ and 
$$
\varepsilon F\partial(\sigma\times c)=F\varepsilon\partial(\sigma\times c)=0
$$
for $d=1$. This establishes points \ref{chain1} and \ref{chain2} of the lemma.% \ref{vanishlemma}.

%Finally, for $d=1$ we must have $|\sigma|=1$ and $|c|=0$ and it is enough to check for $\sigma=[v,w]$ and $c$ a vertex (with a positive sign in $C_0(X_{[v,w]})$). Then $\partial([v,w]\times c)=v\times c_v-w\times c_w$. Hence $F(\partial(\sigma\times c))=\phi(c_v)-\phi(c_w)$. The augmentation map $\varepsilon$ sends this to zero since $c_v$ and $c_w$ are points in the $G$-cover of $X$ that can be connected by a path.
%\begin{remark}
%There is probably a simpler thing to say for this augmentation argument.
%\end{remark}
%Hence, since $BG$ is aspherical and $C_{d-|\sigma|}(X_{\sigma})$ is a free module, $F(\partial(\sigma\times-))$ lifts to $C_d(BG)$. 

%Finally, let $F(\sigma\times-)$ be a lift of $F(\partial(\sigma\times-))$ that vanishes on boundaries if $d-|\sigma|=k$ and vanishes on all chains if $d-|\sigma|>k$. Such $F$ satisfies the properties in the Proposition. This finishes the proof of the inductive step. 

%So, if $|c|=k$ and $c$ is a boundary, or if $|c|>k$, we may define the lift by $F(\sigma\times c)=0$. This finishes the proof. 
\end{proof}

\section{Proof of Theorem \ref{maptheorem}\label{proofoftheorem}}
%Finally, $F$ is chain homotopic to $\phi\circ\pi:C_*(\Delta X)\rightarrow C_*(X)\rightarrow C_*(BG)$ because of the following lemma.  %$\Phi$ constructed in the same way from $\phi$ (instead of the chain maps produced) because $BG$ is aspherical. (See the lemma  below for a proof.) But, this is a contradiction since $I$ induces a non-trivial map on $H^d(-;V)$ for some $\mathbb K[G]$-module $V$, as $G$ has $\mathbb K$-cohomological dimension $d.$ 
%We will need the following Lemma to compare $F$ and $\phi$. 
In order to prove Theorem \ref{maptheorem} we need to compare two chain maps, the original map $\phi:C_*(X)\rightarrow C_*(BG)$ and the map $F:C_*(X')\rightarrow C_*(BG)$ constructed in the previous section (recall that $F$ vanishes on $C_{>\dim(Z)+k}$). 

First, we use the asphericity of $BG$ to prove: 
\begin{lemma}
\label{chainhomotopiclemma}
Suppose $f$ and $g$ are two chain maps $C_*(Y)\rightarrow C_*(BG)$ of augmented chain complexes that agree on $C_{-1}$. Then $f$ and $g$ are chain homotopic. 
\end{lemma}
\begin{proof}
We have $f_{-1}-g_{-1}=0$. To construct a chain homotopy, we need to define $h_{\geq 0}$ in such a way that $f_0-g_0=\partial h_0$ and for each $i>0$ we have $f_i-g_i=\partial h_i+h_{i-1}\partial$. 

To that end, note that $\varepsilon(f_0-g_0)=(f_{-1}-g_{-1})\varepsilon=0$ implies (since $C_0(Y)$ is free and $BG$ is aspherical) that we can define $h_0$ as a lift of $f_0-g_0$.

Suppose we have constructed $h_{\leq i}$. % satisfying $f_i-g_i=\partial h_i+h_{i-1}\partial$. 
Then 
$$
\partial(f_{i+1}-g_{i+1}-h_i\partial)=(f_i-g_i-\partial h_i)\partial=\left\{\begin{array}{ccc}h_{i-1}\partial^2&\mbox{ if }&i>0,\\(f_0-g_0-\partial h_0)\partial &\mbox{ if }& i=0, 
\end{array}\right.
$$
and in both cases the right hand side is equal to zero. 
So (since $C_{i+1}(Y)$ is free and $BG$ is aspherical) we can define $h_{i+1}$ as a lift of $f_{i+1}-g_{i+1}-h_i\partial$.
\end{proof}
\begin{remark}%[Induced maps of chain complexes]
Two maps $\varphi$ and $\psi$ of spaces $X\rightarrow BG$, induce chain maps (of augmented complexes) $C_*^{\varphi}(X)\rightarrow C_*(BG)$ and $C_*^{\psi}(X)\rightarrow C_*(BG)$ that are the identity on $C_{-1}$. If the domains are the same (that is, if $\varphi$ and $\psi$ induce the same $G$-covers of $X$) then Lemma \ref{chainhomotopiclemma} implies the induced maps of chain complexes are chain homotopic. 
\end{remark}
\begin{remark}
The combination of Lemmas \ref{vertexchainmap} and \ref{chainhomotopiclemma} implies Theorem \ref{essentialtheorem}.
\end{remark}

%\subsection*{Proof of Theorem \ref{maptheorem}}
Now, we are ready to put together our results to prove Theorem \ref{maptheorem}.

The map $p:X\rightarrow Z$ is a PL map of simplicial complexes, so, replacing $X$ and $Z$ by appropriate subdivisions, we may assume that $p$ is a simplicial map. For a simplex $\sigma$ in $Z$, pick $z\in\sigma^\circ$ and put $X_\sigma=p^{-1}(z)$ (over the open simplex $\sigma^\circ$, $p$ is a trivial fibration, so different choices of $z$ give homeomorphic fibres). By Proposition \ref{PLmap}, $p$ is a cellular singular fibration. By hypothesis, each $X_{\sigma}$ admits a cell structure with at most $n$ cells of dimension $k$. By Proposition \ref{cellularprop}, we can use these cell structures to construct a cell complex $X'$ together with homotopy equivalences $R:X'\rightarrow X$, and $L:X\rightarrow X'$ such that 
\begin{itemize}
\item 
$X'$ has a filtration by subcomplexes $X_0=X_0'\subset X_1'\subset\dots$,
\item
$X_0'=X_0=\coprod_{v}v\times X_v$,
\item 
$X_i'\setminus X'_{i-1}=\coprod_{|\sigma|=i}\sigma^{\circ}\times X_{\sigma}$, and
\item
$R:X'\rightarrow X$ is a homotopy equivalence relative to $X_0=X_0'$ and $L:X\rightarrow X'$ is a homotopy inverse of $R$ relative to $X_0$. 
\end{itemize}
Now, we use the reference map $\phi R:X'\rightarrow BG$ to construct the cellular chain complex with $\mathbb K[G]$-coefficients on $X'$. Since each $X_v$ has at most $n$ cells in degree $k$, Lemma \ref{vertexchainmap} constructs a chain truncation $f_v:C_*(X_v)\rightarrow C_*(BG)$ that agrees with the chain map induced by a cellular approximation $\phi_v$ of $\phi R\mid _{X_v}$ below degree $k$ and vanishes above degree $k$. Using the fact that each $X_{\sigma}$ has at most $n$ $k$-cells, Proposition \ref{assembleprop} assembles the maps $f_v$ into a chain map $F:C_*(X')\rightarrow C_*(BG)$ that vanishes above dimension $k+\dim(Z)$. Precomposing it with a chain map induced by $L:X\rightarrow X'$, we obtain a chain map $FL:C^{\phi RL}_*(X)\rightarrow C_*(BG)$ that vanishes above dimension $k+\dim(Z)$. 
%So, as $L:X\rightarrow X'$ is a homotopy inverse to $R$ that restricts to the identity on $X_0$, $FL:C_*(X)\rightarrow C_*(BG)$ is a chain map that vanishes above dimension $k+\dim(Z)$. 
Here, $C^{\phi RL}_*(X)$ is the chain complex defined by the cell structure given by the triangulation of $X$ and the reference map $\phi RL$. This reference map is different from $\phi:X\rightarrow BG$. However, since $RL$ is homotopic to $\mathrm{id}_X$ via a homotopy that does not move $X_0$, the maps $\phi$ and $\phi RL$ define the same $G$-covers of $X$, and hence $C^{\phi}_*(X)=C_*^{\phi RL}(X)$. So, $FL$ and $\phi$ are chain maps defined on the same chain complex, and we can compare them via Lemma \ref{chainhomotopiclemma}. Since both are the identity on $C_{-1}$, we conclude that they are chain homotopic. So, we have shown that $\phi$ is chain homotopic to a chain map whose image is zero above dimension $k+\dim(Z)$. This finishes the proof of Theorem \ref{maptheorem}.

\section{Examples}
We end this paper with several examples illustrating phenomena that occur when the injectivity radius is uniformly bounded. %we do not impose assumptions of large injectivity radius.  %illustrative examples %addressing the difference between triangulations and cell structures, what form a topological waist inequality can and cannot take.

\subsection{Large volume/small fibers\label{largevol}} First off, if one has a hyperbolic manifold $M$ with non-trivial $H^1(M)$ (e.g. if $M$ is orientable and has an orientable non-separating hypersurface), then one can use the corresponding map $M\rightarrow S^1$ to construct cyclic covers of $M$. Let $S^1(r)\rightarrow S^1$ be the $r$-fold cyclic cover and $M(r)\rightarrow M$ its pullback to $M$. Then $M(r)\rightarrow S^1(r)$ has the same fibers as $M\rightarrow S^1$. Squashing $S^1(r)$ to $[0,1]$ via an at most two-to-one map, one gets $F:M(r)\rightarrow S^1(r)\rightarrow [0,1]$ with uniformly bounded fibers (independent of $r$), and volume $r \times \mbox{vol}(M)$. 
\subsection{Many simplices/few cells\label{simpvscell}} There are hyperbolic $3$-manifolds as above such that $M^3\rightarrow S^1$ is a fibre bundle. Suppose the fibre is a genus $g$ surface $S_g$. Then $M$ is a mapping torus $T(f)$ of a surface homeomorphism $f:S_g\rightarrow S_g$ and $M(r)$ is the mapping torus of the $r$-fold iterate $f^r$. Pick a cell structure on $S_g$ and a cellular map $h(r)$ homotopic to $f^r$. Then the mapping torus $T(h(r))$ of $h(r)$ is a cell complex homotopy equivalent to $M(r)$ with twice as many cells as $S_g$ (it has a $(k+1)$-cell and a $k$-cell for each $k$-cell of $S_g$). So, $T(h(r))$ is a homotopy model for $M(r)$, there is a uniform bound on the number of cells of $T(h(r))$ that is independent of $r$, 
but the minimal number of simplices in any triangulation of $M(r)$ grows with $r$ since the simplicial volume of the hyperbolic manifold $M(r)$ is proportional to $\mathrm{vol}(M(r))$, which goes to infinity.

\subsection{Large $\pi_1$-rank/small fibers\label{largerank}} We can improve on the first example when there is a $\pi_1$-surjective map $M\rightarrow S^1 \vee S^1$ to a wedge of two circles (which is often true for hyperbolic manifolds \cite{lubotzky}). In that case, take a ``thin'' $r$-fold cover $X(r)\rightarrow S^1 \vee S^1$, e.g. a cover that unwraps one circle but not the other. Such an $X(r)$ maps to $[0,1]$ with fibers that are at most six-to-one. So, if $M(r)\rightarrow M$ is the corresponding cover, then $F:M(r)\rightarrow X(r)\rightarrow [0,1]$ has uniformly bounded fibers, and the fundamental group of $M(r)$ maps onto the fundamental group of $X(r)$, a free group of rank $r+1$, so $M(r)$ has more than $r$ $1$-cells.

\subsection{Maps to $\mathbb R^2$\label{complex}} A variation of this also works for some complex hyperbolic $4$ or $6$-manifolds. Some of these have holomorphic, $\pi_1$-surjective maps to higher genus surfaces $M\rightarrow S_g$ (see \cite{deraux}). An $r$-fold cover $S_g(r)\rightarrow S_g$ is a surface so it maps to the square $[0,1]^2$ via an at most two-to-one map. So, the composition $M(r)\rightarrow S_g(r)\rightarrow [0,1]^2$ has uniformly bounded fibers and $\pi_1$-rank at least twice the genus of the surface $S_g(r)$. In the $\dim M=4$ case, $b_2(M(r))$ also grows linearly in the volume, since $\chi(M^4)$ is a positive integer and hence $\chi(M(r))=r\times \chi(M)\geq r$.

\subsection{On topological waist inequalities} Examples \ref{largevol} and \ref{largerank} tell us that in the hyperbolic setting there can be no topological waist inequalities (for PL maps $F:M^d\rightarrow \mathbb R^m$ from a hyperbolic manifold) of the form
$$
\max_{z\in\mathbb R^m} \mbox{(simplicial volume of } F^{-1}(z)) > C \times\mbox{(simplicial volume of } M)^D,
$$
$$
\max_{z\in\mathbb R^m} \mbox{(number of cells of } F^{-1}(z)) > C\times\mbox{(number of cells of } M)^D
$$
with positive constants $C,D$ that do not depend on $M$ but only on the dimensions $d$ and $m$.
\begin{remark}
For higher rank irreducible lattices there are no constructions analogous to \ref{largevol}-\ref{complex} because of the Margulis normal subgroup theorem. 
\end{remark}
On the other hand Gromov \cite{gromovexpanders} (for $k=1$) and the present paper (for $0<k<d-m$) imply for a PL map $F:M^d\rightarrow \mathbb R^m$ from a hyperbolic $d$-manifold that
$$
\max_{z\in\mathbb R^m} \mbox{(number of }k\mbox{-cells of } F^{-1}(z)) > C \times {\mathrm{inj}(M)\over\log(\mathrm{inj}(M))}.
$$
Conjecturally (see 2.4 of \cite{gromovexpanders}), the right hand side can be replaced with $C\cdot e^{D\cdot \mathrm{inj}(M)}$.

\bibliography{trees}
\bibliographystyle{amsplain}

\end{document}